\documentclass[reqno,11pt,letterpaper]{amsart}
\usepackage{amsmath,amssymb,amsthm,graphicx,mathrsfs,url,dsfont,enumitem}
\usepackage[usenames,dvipsnames]{color}
\usepackage[colorlinks=true,linkcolor=Red,citecolor=Green,urlcolor=Blue]{hyperref}
\usepackage{dsfont}
\usepackage[super]{nth}
\usepackage[open, openlevel=2, depth=3, atend]{bookmark}
\hypersetup{pdfstartview=XYZ}
\usepackage{epigraph}
\usepackage{mathtools}
\usepackage{enumitem}

\def\?[#1]{\textbf{[#1]}\marginpar{\Large{\textbf{??}}}}

\renewcommand{\tilde}{\widetilde}          
\DeclareMathSymbol{\leqslant}{\mathalpha}{AMSa}{"36} 
\DeclareMathSymbol{\geqslant}{\mathalpha}{AMSa}{"3E} 
\DeclareMathSymbol{\eset}{\mathalpha}{AMSb}{"3F}     
\renewcommand{\leq}{\;\leqslant\;}                   
\renewcommand{\geq}{\;\geqslant\;}                   
\renewcommand{\d}{\mathrm{d}}             

\numberwithin{equation}{section}

\newtheorem{theorem}{Theorem}[section]
\newtheorem{lemma}[theorem]{Lemma}

\theoremstyle{remark}

\theoremstyle{definition}

\newcommand{\C}{\mathbb{C}}
\newcommand{\D}{\mathbb{D}}
\newcommand{\R}{\mathbb{R}}
\newcommand{\Z}{\mathbb{Z}}
\renewcommand{\H}{\mathbb{H}}
\newcommand{\N}{\mathbb{N}}

\newcommand{\E}{\mathbb{E}}
\renewcommand{\P}{\mathbb{P}}
\renewcommand{\S}{\mathbb{S}}

\newcommand{\cE}{\mathcal{E}}
\newcommand{\cC}{\mathcal{C}}
\newcommand{\cH}{\mathcal{H}}

\newcommand{\cN}{\mathcal{N}}

\newcommand{\ind}{\mathds{1}}

\newcommand{\cD}{\mathcal{D}}

\newcommand{\laweq}{\overset{\text{law}}{=}}
\newcommand{\cI}{\mathcal{I}}

\newcommand{\norm}[1]{\left\Vert #1\right\Vert}

\author{Guillaume Baverez}
\thanks{Supported by the EPSRC grant EP/L016516/1 for the University of Cambridge CDT, the CCA}
\email{gb539@cam.ac.uk}
\address{Centre for Mathematical Sciences, Wilberforce Road, Cambridge CB30WA, UK}

\title[$\log$-regularity of SLE$_4$]{On the $\log$-regularity of SLE$_4$}

\begin{document}

\begin{abstract}
We prove that the welding homeomomorphism of SLE$_4$ is almost surely $\log$-regular. In a previous version of this work, we had erroneously deduced its removability from this property. Nevertheless, the $\log$-regularity does provide some information and could lead to future developments.
\end{abstract}

\maketitle
\section{Introduction}

	\subsection{Jordan curves}
	
		\subsubsection{Conformal welding and removability}
Let $\eta:\S^1\to\C$ be a Jordan curve, bounding two complementary Jordan domains $\Omega^+,\Omega^-\subset\hat{\C}$. Without loss of generality, we assume that $0\in\Omega^+$ and $\infty\in\Omega^-$. Let $\psi_+:\D^+\to\Omega^+$ (resp. $\psi_-:\D^-\to\Omega^-$) be a Riemann uniformising map fixing $0$ (resp. $\infty$), where $\D^+$ is the unit disc and $\D^-:=\hat{\C}\setminus\bar{\D}^+$. By Carath\'edory's conformal mapping theorem, $\psi_+,\psi_-$ extend continuously to homeomorphisms $\psi_\pm:\bar{\D}^\pm\to\bar{\Omega}^\pm$ and $h:=\psi_-^{-1}\circ\psi_+|_{\S^1}$ is a homeomorphism of the circle called the \emph{conformal welding homeomorphism} of $\eta$. It is well-known that the mapping $\eta\mapsto h$ is neither injective nor onto: namely, there exist distinct curves (viewed up to M\"obius transformations) with the same welding homeomorphism, and not every homeomorphism is the conformal welding of a Jordan curve. At the time of writing, no geometric characterisation of conformal welding homeomorphisms is available in the literature, see \cite{Bishop} for a comprehensive review.

For the curve to be unique, it is sufficient that it is conformally removable (note that the converse is unknown \cite{Younsi18}). Recall that a compact set $K\subset\C$ is \emph{conformally removable} if every homeomorphism of $\hat{\C}$ which is conformal off $K$ is a M\"obius transformation. From the point of view of complex geometry, this means that the conformal maps $\psi_+,\psi_-$ endow the topological sphere $\D^+\sqcup\D^-/\sim_h$ with a well-defined complex structure, where $\sim_h$ is the equivalence relation identifying $x\in\S^1=\partial\D^+$ with $h(x)\in\S^1=\partial\D^-$. Another notion of removability, introduced by Jones \cite{Jones91}, is the removability for (continuous) Sobolev functions, or $H^1$-removability. The set $K$ is \emph{$H^1$-removable} if any $f\in H^1(\C\setminus\eta,\d z)\cap\cC^0(\C)$ belongs to $H^1(\C,\d z)$, that is $H^1(\C\setminus\eta,\d z)\cap\cC^0(\C)=H^1(\C,\d z)\cap\cC^0(\C)$. Jones proved that $H^1$-removability implies conformal removability, but the converse is still an open question. 

	\subsubsection{$\log$-regularity}
The Hilbert space $H^{1/2}(\S^1,\d\theta)$ is the space of traces of $H^1(\D,\d z)$ on $\S^1=\partial\D$. Is a closed subspace of $L^2(\S^1,\d\theta)$ and endowed with the norm $\omega\mapsto\norm{\omega}_{L^2(\S^1,\d\theta)}^2+\norm{\omega}_\partial^2$, where the second term denotes the Dirichlet energy of the harmonic extension of $\omega$ to $\D$. Negligible sets for $H^{1/2}(\S^1,\d\theta)$ are called \emph{polar}, and they are those sets with zero logarithmic capacity. Recall from \cite[Section 3]{Bishop} that a Borel set $E\subset\S^1$ has positive logarithmic capacity if and only if there is a Borel probability measure $\nu$ on $\S^1$ giving full mass to $E$ and with finite logarithmic energy:
\begin{equation}\label{eq:log_energy}
\int_{\S^1\times\S^1}\log\frac{2}{|x-y|}\d\nu(x)\d\nu(y)<\infty.
\end{equation}
 
 Following \cite{Bishop}, we say that $h$ is \emph{$\log$-regular} if $h(E)$ and $h^{-1}(E)$ have zero Lebesgue measure for all polar sets $E\subset\S^1$. In other words, $h$ is $\log$-regular if the pullback measure $\mu:=h^*\d\theta$ does not charge any polar sets of $\S^1$ (and similarly for the pushforward measure), i.e. $\mu$ is a Revuz measure. From the theory of Dirichet forms (see e.g. \cite[Theorem 6.2.1]{Fukushima10}), we can define a Dirichlet form $(\cE,\cD)$ on $L^2(\S^1,\mu)$ with domain
\[\cD:=\{\omega\in L^2(\S^1,\mu),\,\exists\tilde{\omega}\in H^{1/2}(\S^1,\d\theta)\text{ s.t. }\tilde{\omega}=\omega\quad\mu\text{-a.e.}\},\]
 and the form $\cE$ is given unmabiguously by $\cE(\omega,\omega)=\norm{\tilde{\omega}}_\partial^2$. In other words, the injection $H^{1/2}(\S^1,\d\theta)\cap\cC^0(\S^1)\hookrightarrow L^2(\S^1,\mu)$ extends continuously and injectively to $H^{1/2}(\S^1,d\theta)$, and similarly for $H^{1/2}(\S^1,\d\theta)\circ h$ into $L^2(\S^1,\d\theta)$.

There are two natural measures supported on $\eta$: the harmonic measures viewed from 0 and $\infty$ respectively, which we denote by $\sigma_+$ and $\sigma_-$. We say that $\sigma_+$ (resp. $\sigma_-$) is the harmonic measure from the inside (resp. outside) of $\eta$. By conformal invariance, $\sigma_\pm$ is simply the pushforward under $\psi_\pm$ of the uniform measure on $\S^1$. Hence, we can understand the $\log$-regularity of $h$ as the statement that traces of $H^1(\Omega^+,\d z)$ form a closed subspace of $L^2(\eta,\sigma_-)$, and vice-versa. So we can initiate a comparison of these spaces of traces in either $L^2(\eta,\sigma_+)$ or $L^2(\eta,\sigma_-)$. For instance, one can introduce the operator $A:\,H^{1/2}(\S^1,\d\theta)^2\to L^2(\S^1,\d\theta),\,(\omega_+,\omega_-)\mapsto\omega_+-\omega_-\circ h$. This operator encodes the ``jump" accross $\eta$ of a function in $H^1(\C\setminus\eta,\d z)$ whose traces on each side of $\eta$ are given by $\omega_+\circ\psi_+^{-1}$ and $\omega_-\circ\psi_-^{-1}$. In particular, one can expect the kernel of $A$ to contain some information about removability.

	\subsection{Schramm-Loewner Evolution}
	
SLE was introduced by Schramm \cite{Schramm2000} as the conjectured (and now sometimes proved) scaling limit of interfaces of clusters of statistical mechanics models at criticality. These are random fractal curves joining boundary points of simply connected planar domains, characterised by their conformal invariance and domain Markov properties. To each $\kappa\geq0$ corresponds a probability measure SLE$_\kappa$, whose sample path properties depend heavily on the value of $\kappa$. The case $\kappa=0$ is deterministic and corresponds to the (hyperbolic) geodesic flow, while $\kappa>0$ describes random fluctuations around it. A phase transition occurs at $\kappa=4$: the curve is simple for $\kappa\in[0,4]$ but self- and boundary-intersecting for $\kappa\in(4,8)$ \cite[Section 6]{RohdeSchramm05}. Moreover, the Hausdorff dimension of the SLE$_\kappa$ trace a.s. equals $\min(1+\frac{\kappa}{8},2)$ \cite{RohdeSchramm05,Beffara08}.

 A few years after Schramm's groundbreaking paper, it was understood that SLE$_\kappa$ for $\kappa\leq4$ was the Jordan curve arising from the conformal welding of random surfaces according to their boundary length measure \cite{sheffield2016}, the latter being an instance of the ``Liouville measure" \cite{DuplantierSheffield11}. Although the construction of \cite{DuplantierSheffield11} was independent, the Liouville measure is a special case of the ``multiplicative chaos" measures pioneered by Kahane in the 80's \cite{Kahane85}. In \cite{sheffield2016}, Sheffield uses an \textit{a priori} coupling between SLE and the GFF and shows furthermore that the ``quantum lengths" measured from each side of the curve coincide, and correspond to the Liouville measure. This is the ``quantum zipper" theorem, which also states that slicing a random surface with an independent SLE produces two independent random surfaces. Berestycki's review \cite{Berestycki_lqggff} provides a gentle introduction to these topics and an abundance of complementary details. Subsequently, the quantum zipper was systematically used and generalised in the ``mating of trees" approach to Liouville quantum gravity \cite{MatingOfTrees}. The critical Liouville measure ($\kappa=4$) was not constructed when Sheffield's paper was released, but since then Holden \& Powell used recent techniques to extend the result to the critical case \cite{HoldenPowell2018}. 
 
 Another approach to the conformal welding of random surfaces is that of Astala, Kupiainen, Jones \& Smirnov \cite{astala2011}. They use standard complex analysis techniques to show the existence of the welding, but unfortunately the model they consider is not the one that produces SLE. We mention that Aru, Powell, Rohde \& Viklund and the author have ongoing (and independent) works aiming at a construction of SLE via conformal welding of multiplicative chaos without using the coupling with the GFF. We stress that this is not an easy problem since it falls outside of the scope of standard results from the theory of conformal welding.

 Since SLE arises as the interface between conformally welded random surfaces, it is crucial to know that it is conformally removable, as this implies that the complex structure induced on the welded surface is well-defined. It has been known since its introduction that SLE$_\kappa$ is conformally (and $H^1$-) removable for $\kappa<4$ as the boundary of a H\"older domain \cite[Theorem 5.2]{RohdeSchramm05}. However, the case $\kappa=4$ is special as it corresponds to the critical point of the multiplicative chaos measures. At the moment, the only positive result is the one of \cite[Theorem 1.1 \& Section 2]{McenteggartMillerQian}, saying that the only welding satisfying certain geometric conditions is SLE$_4$. On the other hand, it is known that SLE$_4$ is \emph{not} the boundary of a H\"older domain \cite[Section 1.3]{Gwynne18_spectrumSLE}. To our knowledge, it is unknown whether it satisfies the weaker condition on the modulus of continuity contained in \cite[Corollary 4]{JonesSmirnov00}. Motivated by the question of the removability of SLE$_4$ and the considerations of the previous subsection, it is natural to ask whether the welding homeomorphism is $\log$-regular, which we answer affirmatively.
 
\begin{theorem}\label{thm:removable}
Almost surely, the welding homeomorphism of SLE$_4$ is $\log$-regular.
\end{theorem}

This will be proved in Section \ref{subsec:proof}, while Section \ref{sec:log_regular} gives the necessary background.

\textbf{Acknowledgements:} We are grateful to N. Berestycki, C. Bishop, J. Miller, E. Powell and M. Younsi for discussions and comments. We especially thank C. Bishop for pointing out the mistake in the original document and N. Berestycki for motivation and support.

%

\section{Background}\label{sec:log_regular}

	\subsection{Gaussian Multiplicative Chaos}\label{subsec:gmc}
Let $X$ be a centred Gaussian field in the unit interval with covariance
\begin{equation}\label{eq:correl_exact}
\E[X(x)X(y)]=2\log\frac{1}{|x-y|},\qquad\qquad x,y\in(0,1).
\end{equation}
Such a process is \textit{a priori} ill-defined because of the logarithmic divergence on the diagonal, but it can be realised as (the restriction to $(0,1)$ of) the trace on $\R$ of the Gaussian Free Field (GFF) in $\H$ with free boundary conditions. With this procedure, we get a distribution in $(0,1)$ which almost surely belongs to $H^{-s}(0,1)$ for all $s>0$.

\emph{Gaussian Multiplicative Chaos} with parameter $\gamma\in(0,2)$ is the random measure $\mu_\gamma$ on $\cI:=[0,1]$ obtained as the weak limit in probability as $\varepsilon\to0$ of the family of measures
\[\d\mu_{\gamma,\varepsilon}(x):=e^{\frac{\gamma}{2}X_\varepsilon(x)-\frac{\gamma^2}{8}\E[X_\varepsilon^2(x)]}\d x,\]
where $(X_\varepsilon)_{\varepsilon>0}$ is a suitable regularisation of $X$ at scale $\varepsilon$ \cite{Berestycki17}. This measure is defined only up to multiplicative constant (since the GFF is only defined up to additive constant), but we can fix the constant by requiring it to be a probability measure (this also fixes the constant of the GFF). The point $\gamma=2$ is critical and the renormalisation procedure above converges to 0 as $\varepsilon\to0$, but there are several (equivalent) renormalisations that give a non-trivial limit $\mu_2$, e.g.
\begin{equation}\label{eq:renormalise}
\d\mu_{2,\varepsilon}(x):=\sqrt{\log\frac{1}{\varepsilon}}e^{X_\varepsilon(x)-\frac{1}{2}\E[X_\varepsilon^2(x)]}\d x.
\end{equation}
Here, the (deterministic) diverging factor $\sqrt{\log\frac{1}{\varepsilon}}$ compensates for the decay to zero mentioned above. The topology of convergence is the same as in the subcritical case. This renormalisation was considered in \cite{DRSV14} (the so-called ``Seneta-Heyde" renormalisation), but the critical measure can also be obtained by the ``derivative martingale" approach \cite{DRSV14_mart} or as a suitable limit of subcritical measures \cite{APS19}. It is a fact of importance that the limiting measure is universal in the sense that it essentially does not depend on the choice of renormalisation or regularisation of the field \cite{JunnilaSkasman17_uniqueness,Powell18_critical}. For concreteness, we will assume the regularisation $(X_\varepsilon)_{\varepsilon>0}$ of \cite{Barral15}:
\begin{equation}\label{eq:regularised_field}
\frac{1}{2}\E[X_\varepsilon(x)X_\varepsilon(y)]=\left\lbrace
\begin{aligned}
&\log\frac{1}{|x-y|}\qquad&\text{ if }\varepsilon\leq|x-y|\leq1\\
&\log\frac{1}{\varepsilon}+1-\frac{|x-y|}{\varepsilon}\qquad&\text{ if }|x-y|<\varepsilon.
\end{aligned}\right.
\end{equation}

Because of the exact logarithmic form of the covariance \eqref{eq:correl_exact}, the measures $\mu_\gamma$ (for $\gamma\in[0,2]$) satisfy an exact scale invariance property \cite[Appendix A.1]{Barral15}. In particular, for any interval $I\subset\cI$, the restriction $\mu_\gamma|_I$ of $\mu_\gamma$ to $I$ satisfies:
\begin{equation}\label{eq:exact_scale}
\mu_\gamma|_I\laweq|I|e^{\frac{\gamma}{2}X_I-\frac{\gamma^2}{8}\E[X_I^2]}\mu_\gamma^I,
\end{equation}
where $X_I\sim\cN(0,2\log\frac{1}{|I|})$ and $\mu_\gamma^I$ is an independent measure with law $\mu_\gamma^I(\cdot)\laweq\mu_\gamma(|I|^{-1}\,\cdot)$.

For the reader's convenience, we recall some basic properties of these measures, highlighting the pathologies arising at the critical point. The behaviour of $\mu_\gamma$ gets wilder as $\gamma$ increases: almost surely, it gives full mass to a set of Hausdorff dimension $1-\frac{\gamma^2}{4}$, consisting of those points where $X$ is exceptionally large. In the critical case, $\mu_2$ gives full mass to a set of Hausdorff dimension 0, corresponding to the ``maximum" of $X$. This set is still large enough for $\mu_2$ to be non-atomic, see also \cite[Theorem 2]{Barral15} for bounds on the modulus of continuity of $\mu_2$. 

As a result, the distribution of $\mu_\gamma(\cI)$ has a heavy tail near $\infty$, so that positive moments $\E[\mu_\gamma(\cI)^p]$ are finite if and only if $p<\frac{4}{\gamma^2}$. In particular, $\mu_2(\cI)$ does not have a finite expected value, see also \cite[Theorem 1]{Barral15} for precise tail asymptotics. On the other hand, the tail of $\mu_\gamma(\cI)$ at $0^+$ is nice, and $\E[\mu_\gamma(\cI)^p]<\infty$ for all $p<0$ and $\gamma\in[0,2]$.

Let $\mu_+^\gamma,\mu_-^\gamma$ be independent GMCs on $\cI$ with parameter $\gamma\in[0,2]$. Since a.s. $\mu_\pm^\gamma$ is non-atomic and $\mu_\pm^\gamma(\cI)<\infty$, we can define homeomorphisms $h_\pm$ of $\cI$ by $h_\pm^\gamma(x):=\mu^\gamma_\pm[0,x]$. We also set $h:=h_-^{-1}\circ h_+$. For $\gamma<2$, $h_\pm$ and $h^{-1}_\pm$ are a.s. H\"older continuous \cite[Theorem 3.7]{astala2011}, thus so are $h$ and $h^{-1}$ and in particular they preserve polar sets. Hence $h$ is $\log$-regular in the subcritical case. This property is far from clear in the critical case since $h_+$ and $h_-$ are a.s. \emph{not} H\"older continuous. The main result of this section, which is proved in Section \ref{subsec:proof}, is the following theorem.
\begin{theorem}\label{thm:log_regular}
For $\gamma=2$, $h$ is almost surely $\log$-regular.
\end{theorem}

	\subsection{Applications to SLE$_4$ and related models}\label{subsec:applications}
Consider two independent critical GMC measures on $\R$ obtained by exponentiating the trace of the GFF in $\H$. Denote by $h:\R\to\R$ the associated welding homeomorphism. Since countable unions of sets of zero Lebesgue measure have zero Lebesgue measure, we see that $h$ is a.s. $\log$-regular by taking large intervals and applying Theorem \ref{thm:log_regular}.

The construction of SLE using conformal welding is closely related to the above setup and is based on Sheffield's so-called ``$(\gamma,\alpha)$-quantum wedges" \cite[Section 1.6]{sheffield2016}, see also \cite[Section 5.5]{Berestycki_lqggff} and \cite[Section 2.2]{HoldenPowell2018} for the critical case. A quantum wedge is essentially a suitably normalised GFF in $\H$ with free boundary conditions and an extra logarithmic singularity at the origin (parametrised by $\alpha)$. One considers two independent $(\gamma,\gamma)$-quantum wedges and there respective boundary Liouville measure $\mu^\gamma_+,\mu^\gamma_-$ on $\R$, and constructs the homeomorphism $h:\R_+\to\R_-$ characterised by $\mu^\gamma_+[0,x]=\mu^\gamma_-[h(x),0]$ for all $x\in\R_+$. For $\kappa=\gamma^2\in(0,4)$, \cite{sheffield2016} proves that solving the conformal welding problem for this model produces an SLE$_{\kappa=\gamma^2}$ on top of an independent $(\gamma,\gamma-\frac{2}{\gamma})$-quantum wedge. \cite[Theorem 1.2]{HoldenPowell2018} extends this result to $\kappa=\gamma^2=4$. The extra $\log$-singularity at the origin amounts in conditioning the origin to be a typical point of the Liouville measure, so it does not change any capacity properties of the homeomorphism. Thus, Theorem \ref{thm:log_regular} implies that the welding homeomorphism of SLE$_4$ is $\log$-regular, from which Theorem \ref{thm:removable} follows.

Finally, by M\"obius invariance, we get similar statements in the disc model. Namely, let $\mu_+,\mu_-$ be critical GMC measures on $\S^1\simeq\R/\Z$ (normalised to be probability measures), obtained by exponentiating the trace on $\S^1$ of two independent free boundary GFFs in $\D$. Let $h:\S^1\to\S^1$ be the associated welding homeomorphism, i.e. $\mu_+[0,\theta]=\mu_-[0,h(\theta)]$ for all $\theta\in\R/\Z$. Theorem \ref{thm:log_regular} implies that $h$ is almost surely $\log$-regular.

	\subsection{Preliminaries}\label{subsec:preliminaries}
Recall the setup: $\mu_+$ and $\mu_-$ are independent critical GMC measures on $\cI$ as defined in Section \ref{subsec:gmc}, $h_\pm(x)=\frac{\mu_\pm[0,x]}{\mu_\pm(\cI)}$ and $h=h_-^{-1}\circ h_+$. By symmetry, to show the $\log$-regularity of $h$, it suffices to prove that $|h(E)|=0$ for all polar sets $E\subset\cI$.

It is known that $\mu_+$ is a.s. a Revuz measure \cite[Section 4]{RhodesVargas15_criticalLBM} but this is not sufficient to establish that $|h(E)|=0$ for all $E$ polar. Indeed, it could happen (and it actually does) that there exists some polar set $E$ such that $h_+(E)$ has positive Hausdorff dimension, and then nothing could provide \textit{a priori} $h(E)=h^{-1}_-(h_+(E))$ from having positive Lebesgue measure. To prove Theorem \ref{thm:log_regular}, we will have to analyse better the properties of $h_-^{-1}$ and the sets where $\mu_+$ is exceptionally large.
 
For each $n\in\N$, we let 
\[\cD_n:=\left\lbrace[k2^{-n},(k+1)2^{-n}),\,k=0,...,2^n-1\right\rbrace\]
 be the set of all dyadic intervals of length $2^{-n}$, and for $x\in\cI$, $I_n(x)\in\cD_n$ is the dyadic interval containing $x$. A \emph{gauge function} is a non-decreasing function $f:[0,1)\to\R_+$ such that $f(0)=0$. Given such a function, we introduce the set
\[E_n^f:=\{I\in\cD_n|\,\mu(I)\geq f(|I|)\}\]
and $E^f:=\underset{n\to\infty}{\limsup}\,E_n^f$. For $\alpha\geq0$, we denote by $\cH_\alpha$ the $\alpha$-Hausdorff measure, i.e.
\[\cH_\alpha(E)=\underset{\delta\to0}{\lim}\inf\sum_i|I_i|^\alpha,\]
where for a given $\delta>0$ the infimum runs over all coverings of $E$ by countable collections of open intervals $(I_i)$ with Lebesgue measure $|I_i|\leq\delta$. We denote by $\dim E$ the Hausdorff dimension of a set $E\subset\cI$, i.e. $\dim E=\sup\{\alpha\geq0\text{ s.t. }\cH_\alpha(E)=\infty\}=\inf\{\alpha\geq0\text{ s.t. }\cH_\alpha(E)=0\}$.

 In \cite{Barral15}, the authors show that $\mu_+$ gives full mass to a set of Hausdorff dimension zero, i.e. they find a gauge function $f$ such that $\cH_\alpha(E^f)=0$ for all $\alpha>0$ (i.e. $\dim E^f=0$) and $\cH_1(h_+(\cI\setminus E))=0$ (Theorem 4 \& Corollary 24). On the other hand, they give a bound on the modulus of continuity of $h_+$ (Theorem 2), i.e. they find $f$ such that $E^f=\emptyset$. Such $f$'s are given by $f(u)=C(\log\frac{1}{u})^{-k}$ for $k\in(0,\frac{1}{2})$ and some (random) $C>0$. To prove Theorem \ref{thm:log_regular}, we need to investigate in more detail the multifractal properties of $h_-^{-1}$ and the behaviour of $h_+$ on $E^{f_k}$ for $k>\frac{1}{2}$, where here and in the sequel,
 \[f_k(u):=\left(\log\frac{1}{u}\right)^{-k}.\] 

	\section{Proof of Theorem \ref{thm:log_regular}}\label{subsec:proof}

 		\subsection{Bounding the image of $E^{f_k}$}\label{subsubsec:bound_hE}
 The next lemma is a refinement of \cite[Theorem 19 (3)]{Barral14} and its proof follows approximately the same lines.
 
\begin{lemma}\label{lem:dim_Ef}
Fix $k>\frac{1}{2}$. Almost surely, $\dim h_+(E^{f_k})\leq1-\frac{1}{2k}$. 
\end{lemma}
\begin{proof}
Let $\alpha\in(1-\frac{1}{2k},1)$. Fix $1<\beta<\frac{\alpha}{1-\frac{1}{2k}}$ and $0<\varepsilon<\frac{1}{\alpha}(\alpha k-(k-\frac{1}{2})\beta)$. 
This choice of parameters is well-defined and ensures that $\theta:=\alpha k-(k-\frac{1}{2})\beta-\alpha\varepsilon>0$. Also, denote $G^{f_k}:=\underset{n\to\infty}{\liminf}\,(\cI\setminus E^{f_k}_n)$. This is the set of points $x\in\cI$ such that $|h_+(I_n(x))|\leq f_k(x)$ for all sufficiently large $n$. We first aim at showing that $\cH_\alpha(h_+(E^{f_k}\cap G^{f_{k-\varepsilon}}))<\infty$. We have
\begin{equation}\label{eq:hausdorff}
\begin{aligned}
\sum_{I\in\cD_n}|h_+(I)|^\alpha\ind_{\{f_k(|I|)<|h_+(I)|\leq f_{k-\varepsilon}(|I|)\}}
&\leq n^{-\alpha(k-\varepsilon)}\sum_{I\in\cD_n}\ind_{\{f_k(|I|)<|h_+(I)|\leq f_{k-\varepsilon}(|I|)\}}\\
&\leq n^{-\alpha(k-\varepsilon)}\sum_{I\in\cD_n}\left(\frac{|h_+(I)|}{f_k(|I|)}\right)^\beta\\
&=n^{-\theta}\sum_{I\in\cD_n}\left(n^{1/2}|h_+(I)|\right)^\beta.
\end{aligned}
\end{equation}
We need to get a hold on the tail of this last random variable. This is already known for cascades \cite[Lemma 18]{Barral14} and the proof in the case of GMC is a variation of the proof of \cite[Theorem 2]{Barral15} so we will be brief. In the sequel, $X$ denotes the field \eqref{eq:correl_exact} on $(0,1)$ and $\mu=\mu_+$ the associated critical GMC measure. Let $X_{|I|}$ be the regularised field \eqref{eq:regularised_field} and $Y_{|I|}:=X_{|I|}-X_1$. Let $\cD_n^e\subset\cD_n$ be the collection of even intervals, i.e. intervals of the form $[2j2^{-n},(2j+1)2^{-n})$. We have \cite[Equation (28)]{Barral15}
\[\left(\mu(I)\right)_{I\in\cD_n^e}\laweq\left(|I|\int_Ie^{Y_{|I|}-\frac{1}{2}\E[Y_{|I|}^2]}\d\mu_{|I|}\right)_{I\in\cD_n^e},\]
where $\mu_{|I|}$ is independent of $Y_{|I|}$ and the restrictions $(\mu_{2^{-n}}|_I)_{I\in\cD_n^e}$ form a colletion of independent measures. To ease notations, we relabel $Y_{2^{-n}}$ by $Y_n$ and $\mu_{2^{-n}}$ by $\mu_n$. 

Fix $q\in(0,\beta^{-1})$ and denote $S_n:=\sum_{n\in\cD_n^e}(\sqrt{n}\mu(I))^\beta$. It will suffice to get a uniform bound on $\E[S_n^q]$. This will be very similar to step 2 of the proof of \cite[Theorem 2]{Barral15}. First, we rewrite
\begin{equation}\label{eq:gamma}
\E\left[S_n^q\right]=\frac{\Gamma(1-q)}{q}\int_0^\infty\lambda^{-q}\left(1-\E\left[e^{-\lambda S_n}\right]\right)\frac{\d\lambda}{\lambda}.
\end{equation}

Conditionally on $Y_n$, the random variables $(\mu(I))_{I\in\cD_n^e}$ are independent. Moreover, the analysis of \cite{Barral15} shows that 
\begin{equation}\label{eq:prob_mu_large}
\P\left(\mu(I)\geq t|\,Y_n\right)\leq\frac{\cC Z_I}{t},
\end{equation}
 where $Z_I:=\int_Ie^{Y_n-\frac{1}{2}\E[Y_n^2]}\d x$ and $\cC$ is a random variable encapsulating the error. To control this error, \cite{Barral15} conditions on the event that it is not too large and bounds the probability of the complement. We refer the reader to step 3 of their proof for details and assume for now on that $\cC$ is bounded. Using the formula $1-\E[e^{-\lambda X}]=\int_0^\infty\lambda e^{-\lambda x}\P(X\geq x)\d x$, valid for all non-negative random variables $X$, \eqref{eq:prob_mu_large} yields for each $I\in\cD_n^e$:
\begin{align*}
1-\E\left[\left.\exp\left(-\lambda\left(\sqrt{n}\mu(I)\right)^\beta\right)\right|\,Y_n\right]
&=\int_0^\infty\lambda e^{-\lambda t}\P\left(\left.(\sqrt{n}\mu(I))^\beta\geq t\right|\,Y_n\right)\d t\\
&\leq\cC\sqrt{n}Z_I\int_0^\infty\lambda e^{-\lambda t}t^{-1/\beta}\d t=\tilde{\cC}\lambda^{1/\beta}\sqrt{n}Z_I.
\end{align*}
where $\tilde{\cC}:=\Gamma(1-\frac{1}{\beta})\cC$. Thus, by the independence of the measures $(\mu_n|_I)_{I\in\cD_n^e}$ and the inequality $e^{-2x}\leq 1-x$ (valid for $x\in[0,\frac{1}{2}]$), we get for sufficiently small $\lambda>0$
\begin{equation}\label{eq:laplace}\begin{aligned}
1-\E\left[e^{-\lambda S_n}\right]
&=1-\E\left[\prod_{I\in\cD_n^e}\E\left[\left.\exp\left(-\lambda(\sqrt{n}\mu(I))^\beta\right)\right|\,Y_n\right]\right]\\
&\leq1-\E\left[\prod_{I\in\cD_n^e}(1-\tilde{\cC}\lambda^{1/\beta}\sqrt{n}Z_I)\right]\\
&\leq1-\E\left[\exp\left(-2\sum_{I\in\cD_n^e}\tilde{\cC}\lambda^{1/\beta}\sqrt{n}Z_I\right)\right]\\
&\leq1-\E\left[\exp\left(-2\tilde{\cC}\lambda^{1/\beta}\sqrt{n}\int_0^1e^{Y_n-\frac{1}{2}\E[Y_n^2]}\d x\right)\right]
\end{aligned}
\end{equation}
The last term is the Laplace transform of the Seneta-Heyde renormalised measure \eqref{eq:renormalise}, so we can expect to get a uniform bound in $n$. Indeed, from step 4 of the proof of \cite[Theorem 2]{Barral15}, given $\varepsilon\in(0,1-\beta q)$, the last line of \eqref{eq:laplace} is bounded by $C_\varepsilon\lambda^{\frac{1-\varepsilon}{\beta}}$ for some $C_\varepsilon>0$ independent of $n$. Thus, for small $\lambda>0$ we obtain $1-\E\left[e^{-\lambda S_n}\right]\leq C_\varepsilon\lambda^{\frac{1-\varepsilon}{\beta}}$.
Hence, the integrand in the RHS of \eqref{eq:gamma} is $O(\lambda^{\frac{1-\varepsilon}{\beta}-q-1})$ as $\lambda\to0^+$, which is integrable since $\varepsilon<1-\beta q$. This proves that $\E[S_n^q]$ is uniformly bounded as $n\to\infty$. By Markov's inequality and the fact that the law of GMC is the same on even and odd intervals, we get:
\[\P\left(\sum_{I\in\cD_n}\left(\sqrt{n}\mu(I)\right)^\beta\geq n^\frac{\theta}{2}\right)\leq 2\P\left(2S_n\geq n^\frac{\theta}{2}\right)\leq2^{1+q}\E[S_n^q]n^{-\theta q/2},\] 
so the Borel-Cantelli lemma implies that there exists an integer $\ell>\frac{2}{\theta}$ such that almost surely for all $n$ sufficiently large:
\begin{equation}\label{eq:bound_seneta_heyde}
\sum_{I\in\cD_{n^\ell}}\left(n^{\ell/2}\mu(I)\right)^\beta\leq n^{-\frac{\ell\theta}{2}}.
\end{equation}

By definition, for each $N\in\N$, the set $h_+(\cup_{n\geq N}E^{f_k}_n)$ provides a covering of $h_+(E^{f_k})$. Moreover, given $j^\ell\leq n<(j+1)^\ell$ and $x\in\cI$ such that $|h_+(I_n(x))|\geq f_k(2^{-n})$, we have
\[|h_+(I_{j^\ell}(x))|\geq|h_+(I_n(x))|\geq f_k(2^{-n})\geq f_k(2^{-(j+1)^\ell})\sim f_k(2^{-j^\ell}).\]
Hence, the set $\cup_{n^\ell\geq N}E^{cf_k}_{n^\ell}$ provides a covering of $E^{f_k}$ for all $N\in\N$ and $c<1$ (we will use $c=1$ in the sequel for notational simplicity). Intersecting with $G^{f_{k-\varepsilon}}$ and using equations \eqref{eq:hausdorff}, \eqref{eq:bound_seneta_heyde} as well as $\frac{\ell\theta}{2}>1$, we get
\begin{align*}
\cH_\alpha\left(h_+(E^{f_k}\cap G^{f_{k-\varepsilon}})\right)
&\leq\sum_{n\in\N}\sum_{I\in\cD_{n^\ell}}|h_+(I)|^\alpha\ind_{\{f_k(|I|)<h_+(I)\leq f_{k-\varepsilon}(|I|)\}}\\
&\leq\sum_{n\in\N\setminus\{0\}}n^{-\ell\theta}\sum_{I\in\cD_{n^\ell}}\left(n^{\ell/2}|h_+(I)|\right)^\beta<\infty.
\end{align*}
This shows that $\dim h_+(E^{f_k}\cap G^{f_{k-\varepsilon}})\leq\alpha$ almost surely. 

Note that the above argument can also be applied to show that $\dim h_+(E^{f_{k'}}\cap G^{f_{k'-\varepsilon}})\leq\alpha$ for all $k'\leq k$ (with the value of $\varepsilon$ and $\delta$ independent of $k'$). Hence, we get $\dim h_+(E^{f_k})\leq\alpha$ as a finite union of sets of the form $h_+(E^{f_{k-j\varepsilon}}\cap G^{f_{k-(j+1)\varepsilon}})$, $j$ integer, all of which of dimension less than or equal to $\alpha$. This concludes the proof since $\alpha$ can be taken arbitrarily close to $1-\frac{1}{2k}$.
\end{proof}

		\subsection{Properties of the inverse homeomorphism}\label{subsubsec:inverse_homeo}
We turn to the properties of $h_-^{-1}$. We start with an elementary bound on its H\"older regularity.
\begin{lemma}\label{lem:holder_inverse_homeo}
Almost surely, for all $\alpha<\frac{1}{4}$, $h_-^{-1}$ is $\alpha$-H\"older continuous. 

In particular, for all $E\subset\cI$, $\dim E<\frac{1}{4}$ implies $|h_-^{-1}(E)|=0$.
\end{lemma}
\begin{proof}
 For every $\alpha,p>0$ and intervals $I\subset\cI$, we have by Markov's inequality and the exact scale invariance property \eqref{eq:exact_scale}:
\begin{align*}
\P\left(\mu_-(I)\leq|I|^\alpha\right)
=\P\left(\mu_-(I)^{-p}\geq|I|^{-\alpha p}\right)
\leq\E\left[\mu_-(I)^{-p}\right]|I|^{\alpha p}\leq C|I|^{(\alpha-2)p-p^2}.
\end{align*}
Hence, for $\alpha>4$ and $p=1$, we get $\P(\mu_-(I)\leq|I|^\alpha)\leq C|I|^{\alpha-3}=C|I|^{1+(\alpha-4)}$. Specialising to dyadic intervals, the Borel-Cantelli lemma implies that $|h_-(I)|\geq C|I|^\alpha$ for every arc $I\subset\cI$ and some a.s. finite constant $C>0$, i.e. $h_-^{-1}$ is a.s. $\alpha^{-1}$-H\"older continuous.
\end{proof}

Now we investigate the multifractal properties of $h_-^{-1}$ in more detail. Lemma \ref{lem:support_inverse_homeo} below shows that $h_-^{-1}$ transforms a set of Hausdorff dimension $\frac{1}{2}$ into a set of full Lebesgue measure. Intuitively, this can be deduced from the multifractal analysis of $h_-$ as follows. Let 
\[\tilde{E}_\delta:=\left\lbrace x\in\cI:\,\underset{n\to\infty}{\liminf}\frac{\log|h_-^{-1}(I_n(x))|}{\log|I_n(x)|}=\delta\right\rbrace\]
 and $E_\delta$ the analogous set defined for $h_-$ instead of $h_-^{-1}$. Then we expect to have $h_-^{-1}(\tilde{E}_\delta)=E_{1/\delta}$ and $\dim\tilde{E}_\delta=\delta\dim E_{1/\delta}$. To our knowledge, the multifractal analysis of the critical measure has never been written down explicitly, but we can expect $\dim E_\delta=\delta-\frac{\delta^2}{4}$ (hence $\dim\tilde{E}_\delta=1-\frac{1}{4\delta}$) based on known facts from the subcritical case \cite[Section 4]{rhodes2014_gmcReview}. For our purposes, it will be sufficient to give an upper-bound on these dimensions. Notice that these values also explain the H\"older exponent of $h_-^{-1}$ found in Lemma \ref{lem:holder_inverse_homeo}: $E_\delta=\emptyset$ for $\delta>4$, and the local H\"older exponent of $h_-^{-1}$ should be bounded by $\frac{1}{\delta}$ on $h_-(E_\delta)$.

We look for a set which $h_-^{-1}$ maps to a set of full Lebesgue measure, i.e. we look for $\delta$ such that $\delta=\dim\tilde{E}_\delta$. That is, $0=\delta^2-\delta+\frac{1}{4}=(\delta-\frac{1}{2})^2$, i.e. $\delta=\frac{1}{2}$. Precisely, we have:
\begin{lemma}\label{lem:support_inverse_homeo}
Almost surely, $\dim\tilde{E}_{1/2}\leq\frac{1}{2}$ and $|h_-^{-1}(\cI\setminus \tilde{E}_{1/2})|=0$.
\end{lemma}
\begin{proof}
We denote $E_\delta^{\geq}:=\cup_{\delta'\geq\delta}E_{\delta'}$, and their obvious generalisations $E_\delta^{\leq},\tilde{E}_\delta^{\geq},\tilde{E}_\delta^{\leq}$.

For the first claim, note that for all $\varepsilon>0$ and $n\in\N$, the fact that $|h_-^{-1}(\cI)|=1$ implies $\#\{I\in\cD_n,\,|h_-^{-1}(I)|\geq|I|^{1/2+\varepsilon}\}\leq|I|^{-1/2-\varepsilon}$. Thus, for all $\alpha>\frac{1}{2}+\varepsilon$, we have
\begin{align*}
\cH_\alpha(\tilde{E}^{\leq}_{1/2})
&\leq\sum_{n\in\N}\sum_{I\in\cD_n}|I|^\alpha\ind_{\{|h_-^{-1}(I)|\geq|I|^{1/2+\varepsilon}\}}\leq\sum_{n\in\N}|I|^{\alpha-1/2-\varepsilon}<\infty.
\end{align*}
This implies $\dim\tilde{E}_{1/2}\leq\dim\tilde{E}_{1/2}^{\leq}\leq\frac{1}{2}+\varepsilon$, from which the claim follows since $\varepsilon>0$ was arbitrary.

For the second claim, we start with the following observation. Suppose $\tilde{I}\in\cD_n$ is such that $|h_-^{-1}(\tilde{I})|\leq|\tilde{I}|^\delta$. Then there exists $I\in\cD_{\lfloor\delta n\rfloor-1}$ such that $|h_-(I)|\geq|\tilde{I}|\geq(\frac{1}{4}|I|)^{1/\delta}$.  Similarly, if $|h_-^{-1}(\tilde{I})|\geq|\tilde{I}|^\delta$, there exists $I\in\cD_{\lceil\delta n\rceil+1}$ such that $|h_-(I)|\leq|\tilde{I}|\leq(4|I|)^{1/\delta}$. Note that $h_-(I)$ does not cover $\tilde{I}$, but we can simply add the two dyadic intervals in $\cD_n$ directly to the right and to the left of $h_-(I)$. This just has the effect of multiplying everything by a global constant.


Now we get an upper-bound on $\dim E_\delta$. Fix $\delta\in(0,2)$ and set $\eta:=1-\frac{\delta}{2}\in(0,1)$. Let $\alpha>\delta-\frac{\delta^2}{4}$ and $\varepsilon\in(0,2-\delta)$. For all $n\in\N$, we have using exact scale invariance \eqref{eq:exact_scale}:
\begin{align*}
\E\left[\sum_{I\in\cD_n}|I|^\alpha\ind_{\{|I|^{\delta+\varepsilon}\leq|h_-(I)|\leq|I|^{\delta+\varepsilon}\}}\right]
&\leq\E\left[\sum_{I\in\cD_n}|I|^{\alpha}\left(\frac{\mu_-(I)}{|I|^{\delta+\varepsilon}}\right)^\eta\right]\\
&=|I|^{\alpha-(\delta+\varepsilon)\eta-1}\E\left[\mu_-(I)^\eta\right]\\
&\leq C|I|^{\alpha-(\delta+\varepsilon)\eta-1}\times|I|^{2\eta-\eta^2}\\
&\leq C|I|^{\alpha-\varepsilon\eta-(\delta-\delta^2/4)}.
\end{align*}
Summing over $n\in\N$ and taking $\varepsilon>0$ arbitrarily small, we see that $\E[\cH_\alpha(E_\delta)]<\infty$ for all $\alpha>\delta-\frac{\delta^2}{4}$, hence $\dim E_\delta\leq\delta-\frac{\delta^2}{4}$. Moreover, we can write $E_\delta^{\leq}$ as a countable union of sets of zero $\alpha$-Hausdorff measure, hence $\dim E_\delta^{\leq}\leq\delta-\frac{\delta^2}{4}$. A similar argument shows that $\dim E_\delta^{\geq}\leq\delta-\frac{\delta^2}{4}$ for all $\delta\in(2,4)$.

Going back to $h_-^{-1}(\tilde{E}_{1/2})$, let $\varepsilon>0$, $\delta\in(\frac{1}{2},\frac{1}{2}+\varepsilon)$ and $\alpha\in(\frac{1}{\delta}-\frac{1}{4\delta^2},1)$. Using our previous observation, we have

\begin{align*}
\cH_\alpha\left(h_-^{-1}\left(\tilde{E}_{\frac{1}{2}+\varepsilon}^{\geq}\right)\right)
&\leq\sum_{n\in\N}\sum_{\tilde{I}\in\cD_n}|h_-^{-1}(\tilde{I})|^\alpha\ind_{\{|h_-^{-1}(\tilde{I})|\leq|\tilde{I}|^{\delta}\}}\\
&\leq C\sum_{n\in\N}\sum_{I\in\cD_{\lfloor\delta n\rfloor-1}}|I|^\alpha\ind_{\{|h_-(I)|\geq(\frac{1}{4}|I|)^{1/\delta}\}}\\
&\leq C\sum_{n\in\N}\sum_{I\in\cD_n}|I|^\alpha\ind_{\{|h_-(I)|\geq|I|^{1/\delta}\}}.
\end{align*}
Here, $C>0$ is a generic constant that may change from one line to the other. By the above, this last quantity is a.s. finite for our choice of $\alpha$, implying $\cH_1(h_-^{-1}(\tilde{E}_{1/2+\varepsilon}^{\geq}))=0$ almost surely. A similar computation shows that $\cH_1(h_-^{-1}(\tilde{E}_{1/2-\varepsilon}^{\leq}))=0$. Taking a sequence $\varepsilon_n\to0$, we get $|h_-^{-1}(\cI\setminus\tilde{E}_{1/2})|=0$ as a countable union of sets of zero Lebesgue measure.
\end{proof}

		\subsection{Concluding the proof}\label{subsubsec:conclusion}
 The next and final lemma gives an upper-bound on the size of $h_+(E^{f_k})\cap\tilde{E}_{1/2}$.
\begin{lemma}\label{lem:dim_intersection}
Fix $k>\frac{1}{2}$. Almost surely, $\dim(h_+(E^{f_k})\cap\tilde{E}_{1/2})\leq\frac{1}{2}-\frac{1}{2k}$.
\end{lemma}

\begin{proof}
Let $\alpha,\eta>0$ and fix $\varepsilon,\delta>0$ as in the proof of Lemma \ref{lem:dim_Ef} and recall that for all $n\in\N$ we have $M_n:=\#\{I\in\cD_n,\,|h_-^{-1}(I)|\geq|I|^{1/2+\eta}\}\leq|I|^{-1/2-\eta}$. Thus, for all $I\in\cD_n$, we have $\P(|h_-^{-1}(I)|\geq|I|^{1/2+\eta})=\frac{1}{\#\cD_n}\E[M_n]\leq|I|^{1/2-\eta}$. Using this estimate and the fact that $h_+$ is independent of $h_-$, we can condition on $h_+$ to get:
\begin{align*}
&\E\left[\left.\sum_{I\in\cD_n}|h_+(I)|^\alpha\ind_{\{f_k(|I|)<|h_+(I)|\leq f_{k-\varepsilon}(|I|)\}}\ind_{\{|h(I)|\geq|h_+(I)|^{1/2+\eta}\}}\right|\,h_+\right]\\
&\qquad=\sum_{I\in\cD_n}|h_+(I)|^\alpha\ind_{\{f_k(|I|)<|h_+(I)|\leq f_{k-\varepsilon}(|I|)\}}\P\left(\left.|h_-^{-1}(h_+(I))|\geq|h_+(I)|^{1/2+\eta}\right|\,h_+\right)\\
&\qquad\leq C\sum_{I\in\cD_n}|h_+(I)|^{\alpha+\frac{1}{2}-\eta}\ind_{\{f_k(|I|)<|h_+(I)|\leq f_{k-\varepsilon}(|I|)\}}.
\end{align*}

Suppose $\alpha>\dim\,(h_+(E^{f_k}))-\frac{1}{2}+\eta$. From the proof of Lemma \ref{lem:dim_Ef}, there a.s. exists an integer $\ell$ such that 
\begin{equation}\label{eq:bound_hausdorff_intersection}
\sum_{n\in\N}\sum_{I\in\cD_{n^\ell}}|h_+(I)|^{\alpha+\frac{1}{2}-\eta}\ind_{\{f_k(|I|)<|h_+(I)|\leq f_{k-\varepsilon}(|I|)\}}<\infty.
\end{equation}
Moreover, we can cover $h_+(E^{f_k})\cap\tilde{E}_{1/2}$ with the union over $n\in\N$ of those $I\in\cD_{n^\ell}$ such that $|h_+(I)|\in[f_{k-\varepsilon}(|I|),f_k(|I|))$ and $|h(I)|\geq|h_+(I)|^{1/2+\eta}$. We deduce that, almost surely, $\E[\cH_\alpha(h_+(E^{f_k})\cap\tilde{E}_{1/2})|h_+]$ is bounded above by \eqref{eq:bound_hausdorff_intersection}, hence a.s. $\dim\,(h_+(E^{f_k})\cap\tilde{E}_{1/2})\leq\alpha$. Taking $\eta$ arbitrarily close to $0$ enables to take $\alpha$ arbitrarily close to $\dim\,(h_+(E^{f_k}))-\frac{1}{2}$, so that by Lemma \ref{lem:dim_Ef}:
\[\dim\,(h_+(E^{f_k})\cap\tilde{E}_{1/2})\leq\dim h_+(E^{f_k})-\frac{1}{2}\leq\frac{1}{2}-\frac{1}{2k}.\]
%
%
\end{proof}

We now have all the necessary ingredients to conclude the proof of Theorem \ref{thm:log_regular}. 

Fix $k>4$. By definition, the local H\"older regularity of $h_-^{-1}$ on $\tilde{E}_{1/2}$ is $\frac{1}{2}$, so if $F\subset\tilde{E}_{1/2}$, we have $\dim h_-^{-1}(F)\leq 2\dim F$. Hence, Lemma \ref{lem:dim_intersection} implies 
\[\dim\,(h_-^{-1}(h_+(E^{f_k})\cap\tilde{E}_{1/2}))\leq 2\dim(h_+(E^{f_k})\cap\tilde{E}_{1/2})\leq1-\frac{1}{k}<1.\] 
Moreover, by Lemma \ref{lem:support_inverse_homeo}: 
\[|h_-^{-1}(h_+(E^{f_k})\cap(\cI\setminus\tilde{E}_{1/2}))|\leq|h^{-1}_-(\cI\setminus\tilde{E}_{1/2})|=0.\]
This proves $|h(E^{f_k})|=0$ almost surely, so we need only focus on $\cI\setminus E^{f_k}$.

Let $F\subset\cI\setminus E^{f_k}$ and $\nu$ be a Borel probability measure giving full mass to $h_+(F)$. The pullback measure $h_+^*\nu$ gives full mass to $F$ and since we are on $\cI\setminus E^{f_k}$, for all $k'\in(4,k)$ there is $C>0$ such that
\[\int\log\frac{1}{|x-y|}\d h_+^*\nu(x)\d h^*_+\nu(y)\leq C\int|x-y|^{-1/k'}\d\nu(x)\d\nu(y).\]
 That is, we can bound the $\log$-energy of $h_+^*\nu$ by the $\frac{1}{k'}$-energy of $\nu$. By Frostman's lemma, if $\dim h_+(F)>\frac{1}{k'}$, there exists $\nu$ as above with finite $\frac{1}{k'}$-energy, hence $h_+^*\nu$ has finite $\log$-energy and $F$ is not polar by \eqref{eq:log_energy}. Thus, for every polar set $F\subset\cI\setminus E^{f_k}$, we have $\dim\,h_+(F)\leq\frac{1}{k}$, which further implies $\dim h(F)\leq\frac{4}{k}<1$ by Lemma \ref{lem:holder_inverse_homeo} and our assumption on $k$. Hence $|h(F)|=0$ a.s. for all $F\subset\cI$ polar. By symmetry, this also holds for $h^{-1}$, i.e. $h$ is $\log$-regular.

\bibliographystyle{alpha}
\bibliography{../sle}
\end{document}